\documentclass[12pt]{amsart}
\usepackage[T1]{fontenc}
\usepackage[latin9]{inputenc}
\usepackage{comment}
\usepackage{babel}
\usepackage{verbatim}
\usepackage{mathrsfs}
\usepackage{amstext}
\usepackage{amsthm}
\usepackage{amssymb}
\usepackage{amsfonts}
\usepackage{amsmath}
\usepackage{amscd}
\usepackage{esint}
\usepackage{enumerate}
\usepackage[unicode=true,pdfusetitle,
bookmarks=true,bookmarksnumbered=false,bookmarksopen=false,
breaklinks=false,pdfborder={0 0 1},colorlinks=false]
{hyperref}
\usepackage[margin=2.5cm]{geometry}
\usepackage[foot]{amsaddr}
\usepackage{mathtools}
\usepackage{ dsfont }
\usepackage{color,graphicx}
\usepackage[shortlabels]{enumitem}

\makeatletter
\numberwithin{equation}{section}
\numberwithin{figure}{section}
\theoremstyle{plain}
\newtheorem{thm}{\protect\theoremname}[section]

\theoremstyle{definition}
\newtheorem{rem}[thm]{\protect\remarkname}
\theoremstyle{definition}

\theoremstyle{plain}
\newtheorem{prop}[thm]{\protect\propositionname}
\theoremstyle{plain}
\newtheorem{lem}[thm]{\protect\lemmaname}
\theoremstyle{plain}

\theoremstyle{plain} 
\newtheorem{cor}[thm]{\protect\corollaryname}
\theoremstyle{definition}

\theoremstyle{definition}

\theoremstyle{definition}

\theoremstyle{definition}

\theoremstyle{definition}

\usepackage{tikz}
\usetikzlibrary{shapes,arrows}
\usepackage{verbatim}
\usepackage{amsthm}
\usepackage{amstext}

\DeclareMathOperator{\diam}{diam}

\DeclareMathOperator{\Leb}{Leb}

\DeclareMathOperator{\supp}{supp}

\newcommand{\R}{\mathbb R}

\newcommand{\N}{\mathbb N}

\newcommand{\eps}{\varepsilon}

\newcommand{\mH}{\mathcal{H}}

\newcommand{\udim}{\overline{\dim}_B\,}

\newcommand{\hdim}{\dim_H}

\newcommand{\LIN}{\mathrm{LIN}}
\DeclareMathOperator{\Lip}{Lip}
\DeclareMathOperator{\Per}{Per}
\DeclareMathOperator{\Orb}{Orb}
\DeclareMathOperator{\rank}{rank}
\DeclareMathOperator{\Lin}{Lin}
\DeclareMathOperator{\Gr}{Gr}

\makeatother

\providecommand{\conjecturename}{Conjecture}
\providecommand{\corollaryname}{Corollary}
\providecommand{\definitionname}{Definition}
\providecommand{\examplename}{Example}
\providecommand{\lemmaname}{Lemma}
\providecommand{\problemname}{Problem}
\providecommand{\propositionname}{Proposition}
\providecommand{\remarkname}{Remark}
\providecommand{\theoremname}{Theorem}
\providecommand{\taskname}{Task}

\author[A. \'{S}piewak]{Adam \'{S}piewak}

\address{Institute of Mathematics, Polish Academy of Sciences,	ul.~\'Sniadeckich 8, 00-656 Warszawa, Poland}

\email{ad.spiewak@gmail.com}

\thanks{The author was partially supported by the National Science Centre (Poland) grant 2020/39/B/ST1/02329.}

\begin{document}
	
	\title{On the regularity of time-delayed embeddings with self-intersections}

	\maketitle
	
	\begin{abstract}
		We study regularity of the time-delayed coordinate maps \[\phi_{h,k}(x) = (h(x), h(Tx), \ldots, h(T^{k-1}x))\] for a diffeomorphism $T$ of a compact manifold $M$ and smooth observables $h$ on $M$. Takens' embedding theorem shows that if $k > 2\dim M$, then $\phi_{h,k}$ is an embedding for typical $h$. We consider the probabilistic case, where for a given probability measure $\mu$ on $M$ one allows self-intersections in the time-delayed embedding to occur along a zero-measure set. We show that if $k \geq \dim M$ and $k > \hdim(\supp \mu)$, then for a typical observable, $\phi_{h,k}$ is injective on a full-measure set with a pointwise Lipschitz inverse. If moreover $k > \dim M$, then $\phi_{h,k}$ is a local diffeomorphism at almost every point. As an application, we show that if $k > \dim M$, then the Lyapunov exponents of the original system can be approximated with arbitrary precision by almost every orbit in the time-delayed model of the system. We also give almost sure pointwise bounds on the prediction error and provide a non-dynamical analogue of the main result, which can be seen as a probabilistic version of Whitney's embedding theorem.
	\end{abstract}
	
	\section{Introduction and results}
	
	\subsection{Introduction}
	
	This paper contributes to the theory of time-delayed embeddings of dynamical systems. Consider a dynamical system given by a map $T: X \to X$. In many applications the underlying dynamics $(X,T)$ is not directly known, and one has access only to a time series
	\begin{equation}\label{eq: time series} h(x), h(Tx), \ldots, h(T^m x)
	\end{equation}
	generated by an observable $h : X \to \R$ along an orbit of a point $x \in X$. In such a case one is faced with the problem of modeling the unknown dynamics $(X,T)$ based on the observed data. On the most prominent techniques, commonly applied in nonlinear data analysis, is to embed the one-dimensional time series \eqref{eq: time series} into a higher dimensional space $\R^k$ via the \textit{time delayed embedding}, producing a new sequence
	\begin{equation}\label{eq: embedded time series} y_i = (h(T^i x), h(T^{i+1} x), \ldots, h(T^{i+k-1}x)),\ i=0, \ldots, m-k+1.
	\end{equation}
	If the \textit{delay length} k is large enough, one hopes that the original dynamics can be effectively modeled by the \textit{observed dynamics} $y_{i} \mapsto y_{i+1}$. This idea has led to many applications, see e.g. \cite{sm90nonlinear,hgls05distinguishing, SeaClutter, BioclimaticBuildings,  BaghReddy23}.
	
	Starting with a seminal work of Takens \cite{T81}, a mathematical theory has been developed to provide guarantees on the effectiveness of the time-delayed embedding method. Its main object of interest are properties of the \textbf{$\mathbf{k}$-delay coordinate map} $\phi_{h,k}$ corresponding to an observable $h$, defined as
	\[ \phi_{h,k} : X \to \R^k,\ \phi_{h,k}(x) = (h(x), h(Tx), \ldots, h(T^{k-1}x)). \]
	Note that it transforms an orbit $x, Tx, \ldots, T^m x$ of the original system into the embedded time series \eqref{eq: embedded time series} as $y_i = \phi_{h,k}(T^i x)$. The strongest property that one can expect is that $\phi_{h,k}$ is an isomorphism (in a suitable category depending on the structure of $(X,T)$) between $X$ and $\phi_{h,k}(X)\subset \R^k$. In such a case one can define the \textbf{prediction map} $S_{h,k} : \phi_{h,k}(X) \to \phi_{h,k}(X)$ as $S_h = \phi_{h,k} \circ T \circ \phi_{h,k}^{-1}$ and the system $(\phi_{h,k}(X), S_h)$ is isomorphic to the original system $(X,T)$, as the diagram
	\begin{equation}\label{eq:comm_diagram}
		\begin{CD}
			X @>T>> X\\
			@VV\phi_{h,k} V @VV \phi_{h,k} V\\
			\phi_{h,k}(X) @>S_{h,k}>> \phi_{h,k}(X)
		\end{CD}
	\end{equation}
	commutes. Whenever this happens, one can indeed claim that the observed dynamics $S_h$ gives a reliable model of the original dynamics (note that $S_h(y_i) = y_{i+1}$ for $y_i$ as in \eqref{eq: embedded time series}). Takens \cite{T81} proved that if $M$ is a $C^2$ manifold, then for a (Baire) generic pair $(T,h)$ of a $C^2$ diffeomorphism $T:M \to M$ and a $C^2$ observable $h : M \to \R$, the $k$-delay coordinate map $\phi_{h,k}$ is an embedding of $M$ into $\R^k$, provided that $k > 2\dim M$ (the same is true in the $C^r$ category, $1 \leq r \leq \infty$ \cite{Noakes91}). Later, Sauer, Yorke and Casdagli \cite{SYC91} extended this result to arbitrary compact sets $X \subset \R^N$, by proving that for an injective Lipschitz map $T : X \to X$, a prevalent Lipschitz observable $h : X \to \R$ gives rise to an injective $k$-delay coordinate map $\phi_{h,k}$, provided that $k > 2\udim X$ and $2\udim\left(\Per_p(T) \right)< p$ for $p=1, \ldots, k-1$ (here $\udim$ denotes the upper box-counting dimension and $\Per_p (T)$ is the set of $p$-periodic points).\footnote{in fact, \cite{SYC91} considers $C^1$ diffeomorphisms of an open neighbourhood $U$ of $X$ and $C^1$ observables on $U$. For the presented formulation see \cite[Theorem 4.5]{Rob11}.}	This formulation allows for applications to geometrically complicated fractal sets, e.g. attractors of chaotic dynamical systems. These seminal results are commonly invoked as a mathematical justification for the validity of time-delayed based techniques of modeling dynamics from observed time series, e.g. \cite{SugiharaFishes,hgls05distinguishing,sm90nonlinear}.
	
	This paper deals with a \textit{probabilistic} approach to the Takens theory, where one studies the behaviour of the time-delayed embedding for almost every orbit of the system (rather than on the full space) with respect to a given probability distribution $\mu$. In this setting one can allow for an existence of self-intersections in the embedding (i.e. lack of injectivity of $\phi_{h,k}$) or failure of the prediction (i.e. non-commutativity of \eqref{eq:comm_diagram}), as long as these occur with negligible probability. In this case a smaller delay length $k$ may be sufficient. It was conjectured by Schroer, Sauer, Ott and Yorke \cite{SSOY98} that a delay length $k > \dim \mu$ (as opposed to $k > 2\dim X$) should be typically sufficient for a statistically reliable prediction of the system (we give more details in Section \ref{subsec: pred errors}) This or similar approach was repeated afterwards \cite{Liu10, McSharrySmith04, OrtegaLouis98} and recently studied mathematically, see e.g. \cite{MezicBanaszuk04, BGS20, BGSPredict, KoltaiKunde, BGOMY24}. A basic version of a probabilistic Takens time-delayed embedding theorem was provided by Bara\'nski, Gutman and the author \cite[Theorem 1.2]{BGS20}. It states that if $T:X\to X$ is an injective Lipschitz map on a Borel set $X \subset \R^N$ and $\mu$ is a Borel probability measure on $X$, then for a prevalent Lipschitz observable $h : X \to \R$ there exists a full $\mu$-measure Borel set $X_h \subset X$ such that $\phi_{h,k}$ is injective on $X_h$, provided that $k > \hdim \mu$ and $\mu \left( \bigcup \limits_{p=1}^{k-1} \Per_p(T)\right) = 0$, where $\hdim$ denotes the Hausdorff dimension. The threshold $k > \hdim \mu$ is essentially precise for this problem \cite{BGSLimits}. If $\mu$ is additionally $T$-invariant, then $X_h$ can be taken to satisfy $T(X_h) \subset X_h$. In this case one can define the prediction map
	\begin{equation}\label{eq: prediction map}
	S_{h,k} : \phi_{h,k}(X_h) \to \phi_{h,k}(X_h),\ S_{h,k} = \phi_{h,k} \circ T \circ \phi_{h,k}^{-1},
	\end{equation}
	where $\phi_{h,k}^{-1} : \phi_{h,k}(X_h) \to X_h$ denotes the almost surely defined inverse to $\phi_{h,k}$. Note that in this case the diagram \eqref{eq:comm_diagram} commutes almost surely, hence the observed system $(\phi_{h,k}(X_h), S_{h,k})$ gives a $1-1$ time-delayed model of the dynamics $(X,T)$ on a full $\mu$-measure set. The goal of this note is to study regularity of the inverse $\phi_{h,k}^{-1}$ on $\phi_{h,k}(X_h)$, and in consequence also regularity of $S_{h,k}$. This problem is essential for the quantitative study of time-delayed embeddings with self-intersections. In general, assumption $\hdim \mu < k$ is not sufficient to guarantee that $\phi_{h,k}^{-1}$ is continuous \cite[Corollary 8.2]{BGSPredict} (note that in general we cannot expect $X_h$ to be compact), but $X_h$ can be chosen so that $S_{h,k}$ is continuous on $\phi_{h,k}(X_h)$ if $\supp \mu$ is compact with $\hdim(\supp \mu) < k$ \cite[Theorem 1.21]{BGSPredict}. For many applications continuity is insufficient, as it gives no quantitative control on how measurement errors propagate in the reconstruction and prediction procedures. Moreover, important geometric characteristics of the dynamics, such as Lyapunov exponents, are not guaranteed to be preserved in this case. It turns out that assumption $\hdim(\supp \mu) < k$ is insufficient  for H\"older continuity of $\phi_{h,k}^{-1}$ for any $\beta>0$ \cite[Remark 3.10]{BGS20} or even its pointwise H\"older continuity \cite[Proposition 6.2]{BGSRegularity}, hence we have to consider stronger assumptions on $X$ for improved regularity. We bridge the original Takens' setting of diffeomorphisms of smooth manifolds with the recent probabilistic approach and prove in Theorem \ref{thm: takens main} that if $X$ is a compact subset of a smooth manifold $M$ with $k>\hdim X$ and $k \geq \dim M$, then $\phi_{h,k}$ is pointwise bi-Lipschitz at $\mu$-a.e. $x \in X$ for a prevalent smooth observable. If we take $X=M$, then $\phi_{h,k}$ is additionally a local diffeomorphism at $\mu$-a.e. $x \in X$. The former result allows for an almost sure control of prediction errors as defined in \cite{SSOY98} (see Theorem \ref{thm: prediction error}), while the latter guarantees that the Lyapunov exponents of invariant measures can be approximated up tp an arbitrary precision with the use of the embedded dynamics (Theorem \ref{thm: lyap exp}). We also provide a non-dynamical version of the main result (Theorem \ref{thm: projections main}), which can be seen as a probabilistic counterpart of the celebrated Whitney embedding theorem \cite{Whitney36}. A special case gives this result for almost every orthogonal projection onto $k$-dimensional linear subspaces in $\R^N$, providing a probabilistic embedding theorem for manifolds in the setting of the seminal Marstrand embedding theorem \cite{Marstrand,Mattila-proj} from geometric measure theory, see Theorem \ref{thm: projections orth}. Regularity of random linear projections plays an important role in the theory of analog compression \cite{WV10, LosslessAnalogCompression}.
	
	\subsection{Results}
	We will consider dynamics on a finite-dimensional, compact, Riemannian $C^r$  manifold $M$, allowing any $r \in \{1, 2, \ldots, \infty \}$. We denote by $\rho : M \times M \to [0,\infty)$ the metric on $M$ induced by the Riemannian metric tensor and consider the Hausdorff dimension $\hdim$ of subsets of $M$ with respect to $\rho$. We employ prevalence as the notion of typicality in the space of $C^r$ functions on $M$. See Section \ref{sec: prelim} for details. For a map $T: M \to M$ we denote by $\Per_p (T) = \{ x \in M : T^p x = x\}$ the set of $p$-periodic points.
	
	\begin{thm}\label{thm: takens main}
	Let $M$ be a compact, Riemannian $C^r$ manifold. Let $T : M \to M$ be a $C^r$ diffeomorphism. Let $X \subset M$ be a compact set and $\mu$ be a Borel probability measure on $X$. Fix $k \in \N$ and assume that $\mu \left( \bigcup \limits_{p=1}^{k-1} \Per_p(T) \right) = 0$. The following hold for a prevalent $C^r$ observable $h : M \to \R$
	\begin{enumerate}
	\item\label{it: takens main hdim} if $k > \hdim X$ and $k \geq \dim M$, then for $\mu$-a.e. $x \in X$ there exists $C = C(x, h) \in [0, \infty)$ so that
	\begin{equation}\label{eq: takens bilip main} \rho(x,y) \leq C \|\phi_{h,k}(x) - \phi_{h,k}(y)\| \text{ holds for every } y \in X,
	\end{equation}
	\item\label{it: takens main dim} if $k > \dim M$, then for $\mu$-a.e. $x \in X$ there exist open neighbourhoods $U$ of $x$ in $M$ and $V$ of $\phi_{h,k}(x)$ in $\R^k$ such that $\phi_{h,k}$ is a $C^r$ diffeomorphism between $U$ and $V \cap \phi_{h,k}(M)$.
	\end{enumerate}
	\end{thm}
	Part of the last statement is that $V \cap \phi_{h,k}(M)$ is an embedded $C^r$ submanifold in $\R^k$ (note that the full image $\phi_{h,k}(M)$ is not guaranteed to be a submanifold in $\R^k$). It also follows from the statement that $\phi_{h,k}^{-1}(V) = U$ and so there are no points of (global) self-intersection of $\phi_{h,k}(M)$ in $V$. These remarks apply to other similar statements throughout the paper.

	In a non-dynamical context we can prove the following "probabilistic" version of Whitney's embedding theorem.
	
	\begin{thm}\label{thm: projections main}
	Let $M$ be a compact, Riemannian $C^r$ manifold. Let $X \subset M$ be a compact set and let $\mu$ be a Borel probability measure on $X$. Fix $k \in \N$. The following hold for a prevalent $C^r$ function $\phi : M \to \R^k$:
	\begin{enumerate}
		\item if $k > \hdim X$ and $k \geq \dim M$, then for $\mu$-a.e. $x \in X$ there exists $C = C(x, \phi) \in [0, \infty)$ so that
		\begin{equation}\label{eq: proj bilip main} \rho(x,y) \leq C \|\phi(x) - \phi(y)\| \text{ holds for every } y \in X,
		\end{equation}
		\item if $k > \dim M$, then for $\mu$-a.e. $x \in X$ there exist open neighbourhoods $U$ of $x$ in $M$ and $V$ of $\phi(x)$ in $\R^k$ such that $\phi$ is a $C^r$ diffeomorphism between $U$ and $V \cap \phi(M)$.
	\end{enumerate}
	\end{thm}
	
	In both theorems above it suffices to assume $\mH^k(X) = 0$ instead of $\hdim X < k$, where $\mH^k$ denotes the $k$-dimensional Hausdorff measure on $X$. Moreover, prevalence in Theorem \ref{thm: takens main} is obtained by random polynomial perturbations of $C^r$ observables $h : M \to \R$, while in Theorem \ref{thm: projections main} by random linear perturbation of $C^k$ maps $\phi : M \to \R^k$ (in both cases, after embedding $M$ into $\R^N$). See Theorems \ref{thm: projections detail} and \ref{thm: takens detail} for more detailed statements.
	
	\begin{rem}
	In both Theorems \ref{thm: takens main} and \ref{thm: projections main}, the Riemannian structure is not required for the second point. In the first point, the metric structure is needed for the Hausdorff dimension and the bi-Lipschitz condition to be well defined, but the particular choice of the Riemannian structure does not matter (except for the value of the constant $C(x,\phi)$), as all Riemannian metrics on a compact manifold are bi-Lipschitz equivalent. Since every compact $C^r$ manifold admits a Riemannian structure, this assumption puts no additional restrictions.
	\end{rem}
	
	\section{Preliminaries}\label{sec: prelim}
	
	Below we recall some basics on geometric measure theory, smooth manifolds and prevalence which will be useful throughout the paper.
	
	We shall endow $\R^m$ with the Euclidean norm $\| \cdot \|$ and $B_m(x,r)$ will denote the open $r$-ball in this metric. Given linear spaces $V$ and $W$ we denote by $\LIN(V,W)$ the linear space of linear operators $L : V \to W$. If $V$ and $W$ are finite dimensional, then so is $\LIN(V,W)$ and hence we can consider the Lebesgue measure on it. While it has no canonical normalization, we shall only need the classes of zero and full measure sets, which are determined uniquely. The following basic lemma will be useful for us.
	
	\begin{lem}\label{lem: full rank}
	Let $V, W$ be finite-dimensional linear spaces. Assume that $\dim V \leq \dim W$ and fix $L_0 \in \LIN(V,W)$. Then $\rank(L_0 + L) = \dim V$ for Lebesgue almost every $L \in \LIN(V,W)$.
	\end{lem}
	
	\begin{proof}
	For $\rank(L_0 + L)=\dim V$ it suffices that there exists a $\dim V \times \dim V$ minor of the matrix of $L_0 + L$ (in any fixed linear bases of $V$ and $W$) with non-zero determinant. This holds for Lebesgue almost every $L \in \LIN(V,W)$.
	\end{proof}
	
	\subsection{Hausdorff measure and dimension}
	
	Given a metric space $(X, \rho)$, the \textbf{$\mathbf{s}$-dimensional (outer) Hausdorff measure} $\mH^s(A)$ of a set $A \subset X$ is defined as
	\[ \mH^s(A) = \lim_{\delta \to 0}\ \inf \Big\{ \sum_{i = 1}^{\infty} \diam(U_i)^s : X \subset \bigcup_{i=1}^{\infty} U_i,\ \diam(U_i) \leq \delta  \Big\}\]
	and the \textbf{Hausdorff dimension} of $A$ is
	\[ \hdim(A) = \inf \{ s>0 : \mH^s(A) = 0 \} = \sup \{ s>0 : \mH^s(A) = +\infty \}. \]
	See e.g. \cite{Rob11} for a more detailed introduction to Hausdorff dimension.
	
	\subsection{Smooth manifolds}\label{subsec: manifolds} As a general source on smooth manifolds we refer to \cite{LeeIntroSmoothManifolds}. We also refer the reader to \cite{HirschDiffTop}, which treats the $C^r$ case for $r=1, 2, \ldots, \infty$, and to \cite{N73, SpivakDiffGeom} as additional sources.
	
	We consider only finite dimensional $C^r$ manifolds with $r=1, 2, \ldots, \infty$ and assume them to be Hausdorff. For $x \in M$ we denote by $T_x M$ the tangent space to $M$ at $x$. For a differentiable map $f : M \to N$ between $C^r$ manifolds and $x \in M$, we shall denote by $D_x f : T_x M \to T_{f(x)}N$ the differential of $f$ at $x$. Recall that $T_x M$ is a $\dim M$-dimensional vector space and $D_x f$ is a linear map. For $M = \R^N$ we shall identify $T_x \R^N = \R^N$ via the canonical basis. A differentiable map $f : M \to N$ is called an \textbf{immersion at a point} $x \in M$ if $D_x f$ is injective, and $f$ is called an \textbf{immersion} if it is an immersion at every point of $M$. Moreover, a $C^r$ map $f : M \to N$ is called a \textbf{$\mathbf{C^r}$ diffeomorphism} if it is a homeomorphism whose inverse if also $C^r$, and it is called a \textbf{$\mathbf{C^r}$ embedding} if it is an immersion and a homeomorphism onto its image. We will repeatedly use the following basic fact.
	
	\begin{lem}\label{lem: immersion}
	Let $f : M \to N$ be a $C^r$ map between $C^r$ manifolds. If $f$ is an immersion at $x \in X$ then there exists an open neighbourhood $U$ of $x$ such that $f(U)$ is a $\dim(M)$-dimensional embedded $C^r$ submanifold in $N$ and $f|_U : U \to f(U)$ is a $C^r$ diffeomorphism.
	\end{lem}
	
	The above lemma follows e.g. from \cite[Proposition 4.1, Theorem 4.12, Theorem 5.8]{LeeIntroSmoothManifolds} and is in fact a consequence of the inverse function theorem (see also \cite[Theorem 1.9.(2)]{SpivakDiffGeom} or \cite[Prop. 1.B5]{DontchevRockafellarTyrrell} for the $C^r$ case).
	
	%This is a consequence of the local immersion theorem (see e.g. \cite[Chapter 1.3]{GuileminPollackDiffTop}), which in turn follows from the inverse function theorem (see e.g. \cite[Chapter 1.3]{HirschDiffTop} \cite[Theorem 9.(2)]{SpivakDiffGeom} or \cite[Prop. 1.B5]{DontchevRockafellarTyrrell}). 

	Let $M,N$ be $C^r$ manifolds and assume that $M$ is compact. Let $C^r(M,N)$ denote the set of all $C^r$ maps $f : M \to N$. We endow $C^r(M,N)$ with the topology of uniform convergence of the derivatives of orders up to $r$, see \cite[Chapter 2]{HirschDiffTop} or \cite[Section 2.15]{N73}. This topology is metrizable with a complete metric. In particular, $C^r(M,\R^k)$ is a locally convex complete linear metrizable space (a Fr\'echet space). 
	
	Every compact $C^r$ manifold $M$ can be embedded into $\R^N$ for $N$ large enough, i.e. $M$ is $C^r$ diffeomorphic to a $C^r$ submanifold in $\R^N$, see e.g. \cite[Theorem 3.4]{HirschDiffTop}. By Whitney's theorem, one can take $N\geq 2\dim M + 1$ and in this case the set of $C^r$ embeddings is open and dense in $C^r(M,\R^N)$ (see \cite[Theorem 2]{Whitney36} and \cite[Theorem 2.15.9]{N73}). If $M \subset \R^N$ is an embedded $C^r$ submanifold in $\R^N$, then for every $x\in M$ we can identify $T_x M$ with a $\dim M$-dimensional linear subspace of $\R^N$.
	
	We shall also consider Riemannian manifolds, i.e. $C^r$ manifolds equipped with a metric tensor $g$. For our needs it suffices to assume that the scalar product $\langle \cdot,\cdot \rangle_{g_x}$ on $T_x M$ depends on $x$ in a continuous way. A metric tensor $g$ induces a Riemannian distance (metric) $\rho = \rho_g$ on $M$, which metrizes the topology on $M$ (see \cite[Chapter 13]{LeeIntroSmoothManifolds}). If $M$ is compact, then given any two metric tensors $g_1, g_2$, the induced metrics are bi-Lipschitz equivalent, i.e. there exists $C>0$ such that $\frac{1}{C} \rho_{g_1}(x,y) \leq \rho_{g_2}(x,y) \leq C \rho_{g_1}(x,y) \text{ for } x,y \in M$,	see e.g. \cite[Lemma 2.53]{LeeRiemannian}. Similarly, a $C^r$ embedding $f : M \to \R^k$ of a compact, Riemannian $C^r$ manifold is a bi-Lipschitz map between $M$ endowed with the Riemannian distance and $f(M)$ endowed with the Euclidean distance.
	
	\subsection{Prevalence} Let $V$ be a complete linear metric space. A Borel set $S \subset V$ is called \textbf{prevalent} if there exists a Borel measure $\nu$ in $V$, which is positive and finite on some compact set in $V$, such that for every $v \in V$, there holds $v + e \in S$ for $\nu$-almost every $e \in V$. A non-Borel subset of $V$ is prevalent if it contains a prevalent Borel subset. This notion may be considered as an analogue of `Lebesgue almost sure' condition in infinite dimensional spaces. For more information on prevalence we refer to \cite{Prevalence92} and \cite[Chapter~5]{Rob11}. We shall consider prevalence in the spaces $C^r(M, \R^k)$. Note that if $L : V \to W$ is a linear isomorphism of complete linear metric spaces, then set $S$ is prevalent in $V$ if and only if $L(S)$ is prevalent in $W$. This leads to the following remark.
	
	\begin{rem}\label{rem: embed}
		Let $M$ be a compact $C^r$-manifold and let $\psi : M \to \R^N$ be a $C^r$ embedding. Then $\psi(M)$ is a $C^r$-submanifold in $\R^N$ and $\psi$ induces a linear isomorphism $L : C^r(M, \R^k) \to C^r(\psi(M), \R^k)$ given by $L(f) = f \circ \psi^{-1}$. As by Whitney's theorem every compact $C^r$ manifold $M$ embeds into $\R^N$ for some $N$, we see that in order to prove that for such manifolds a certain property holds for a prevalent $C^r$ function $f : M \to \R^k$, it suffices to consider the case of $C^r$ manifolds embedded in a Euclidean space.
	\end{rem}
	
	\section{Proof of Theorem \ref{thm: projections main}}
	
	We shall prove Theorem \ref{thm: projections main} by establishing the following, more general result.
	
	\begin{thm}\label{thm: projections detail}
		Let $M \subset \R^N$ be a compact, embedded $C^r$ submanifold in $\R^N$. Let $X \subset M$ be a compact set and let $\mu$ be a Borel probability measure on $X$. Fix $k \in \N$ and a $C^r$ function $\phi : M \to \R^k$. The following hold for Lebesgue almost every linear map $L \in \LIN(\R^N, \R^k)$, where write $\phi_L : M \to \R^k$ for $\phi_L := \phi + L$
		\begin{enumerate}
			\item\label{it: proj detail biLip} if $\mH^k(X) = 0$ and $k \geq \dim M$, then for $\mu$-a.e. $x$ there exists $C = C(x, L) \in [0, \infty)$ so that
			\begin{equation}\label{eq: proj biLip} \| x- y \| \leq C \|\phi_L(x) - \phi_L(y)\| \text{ holds for every } y \in X,
			\end{equation}
			\item\label{it: proj detail diffeo} if $k > \dim M$, then for $\mu$-a.e. $x$ there exist open neighbourhoods $U$ of $x$ in $M$ and $V$ of $\phi_L(x)$ in $\R^k$ such that $\phi_L$ is a $C^r$ diffeomorphism between $U$ and $V \cap \phi_L(M)$.
		\end{enumerate}
	\end{thm}
	
	Before proving the theorem, let us note that in the special case $\phi = 0$ we obtain a result for typical orthogonal projections, in the spirit of Marstrand-type projection theorems in geometric measure theory \cite{Marstrand, Mattila-proj, HuTaylorProjections}. Let $\Gr(k,N)$ denote the Grassmannian of $k$-dimensional linear subspaces of $\R^N$, equipped with the $O(N)$-invariant probability measure \cite[Section 3.9]{mattila}. For $V\in \Gr(d,N)$ we denote by $P_V : \R^N \to V$ the orthogonal projection onto $V$.
	
	\begin{thm}\label{thm: projections orth}
		Let $M \subset \R^N$ be a compact, embedded $C^r$ submanifold in $\R^N$. Let $X \subset M$ be a compact set and let $\mu$ be a Borel probability measure on $X$. Fix $k \in \N$. Then the following holds for almost every $V \in \Gr(k,N)$
		\begin{enumerate}
			\item\label{it: proj orth biLip} if $\mH^k(X) = 0$ and $k \geq \dim M$, then for $\mu$-a.e. $x$ there exists $C = C(x, V) \in [0, \infty)$ so that
			\[ \| x- y \| \leq C \|P_V x - P_V y\| \text{ holds for every } y \in X,\]
			\item\label{it: proj orth diffeo} if $k > \dim M$, then for $\mu$-a.e. $x$ there exist open neighbourhoods $U$ of $x$ in $M$ and $V$ of $P_V(x)$ in $\R^k$ such that $P_V$ is a $C^r$ diffeomorphism between $U$ and $V \cap P_V(M)$.
		\end{enumerate}
	\end{thm}
	
	The above result follows from Theorem \ref{thm: projections detail} applied with $\phi=0$ (see \cite[Remark 3.1]{BGSRegularity}).
		
 	\begin{proof}[{\bf Proof of Theorem \ref{thm: projections detail}}]
	The proof combines two elements: results on global almost sure continuity of the inverse to the embedding and its local differentiability at almost every point. The former one follows from \cite{BGSRegularity}, while the latter is obtained as an application of Lemma \ref{lem: immersion} . We explain the details below. \cite[Theorem 4.2]{BGSRegularity} gives the following:\\
	
	\textbf{Fact 1.}:  if $\mH^k(X) = 0$, then for almost every $L \in \LIN(\R^N, \R^k)$, the following holds for $\mu$-a.e. $x \in X$: for every $\eps > 0$ there exists $\delta = \delta(x, L, \eps) > 0$ such that
	\begin{equation}\label{eq: proj cont}
		\text{for every }  y \in X, \text{ if } \|\phi_L(x) - \phi_L(y)\| \leq \delta, \text{ then } \|x - y\| \leq \eps.
	\end{equation}\\

	On the other hand, if $k \geq \dim M$, then Lemma \ref{lem: immersion}  gives that if $\rank(D_x \phi_L) = \dim M$, then there exists an open neighbourhood $U$ of $x$ in $M$ such that $\phi_L$ is a $C^r$ diffeomorphism of $U$ onto its image. For $x \in X$, we have $D_x \phi_L = D_x \phi + L|_{T_x M} \in \LIN(T_x M, \R^k)$ (recall that we assume $M$ to be embedded in $\R^N$, hence we can treat $T_x M$ as a linear subspace of $\R^N$). As $\dim T_x M = \dim M \leq k$, Lemma \ref{lem: full rank} gives that for given $x \in X$, equality $\rank(D_x \phi_L) = \dim M$ holds for Lebesgue almost every $L \in \LIN(\R^N, \R^k)$. By Fubini's theorem and Lemma \ref{lem: immersion} we conclude the following:\\
	
	\textbf{Fact 2.}: if $k \geq \dim M$, then for almost every $L \in \LIN(\R^N, \R^k)$, the following holds for $\mu$-a.e. $x \in X$: 
	\begin{equation}\label{eq: proj diffeo}
	\begin{gathered}
		\text{there exists an open neighbourhood } U \text{ of } x \text{ in } M \text{ such that } \\
		\phi_L \text{ is a } C^r\text{-diffeomorphism between } U \text{ and } \phi_L(U).
	\end{gathered}
	\end{equation}\\
	
	We shall now explain how Theorem \ref{thm: projections detail} follows from Facts 1 and 2. For point \eqref{it: proj detail biLip} we argue as follows. By Fubini's theorem, it suffices to prove that for fixed $x \in X $, for almost every $L \in \LIN(\R^N, \R^k)$ there exists $C = C(x,L)$ such that \eqref{eq: proj biLip} holds. By Fact 2, for almost every $L \in \LIN(\R^N, \R^k)$ there exists an open neighbourhood $U$ of $x$ such that \eqref{eq: proj diffeo} holds, and hence also \eqref{eq: proj biLip} holds for all $y \in U \cap M$. Let $\eps = \eps(x, L)>0$ be such that $\overline{B(x,\eps)} \cap M \subset U$. By Fact 1, for Lebesgue almost every $L \in \LIN(\R^N, \R^k)$ there exists $\delta = \delta(x,L, \eps)>0$ such that \eqref{eq: proj cont} holds for the above choice of $\eps$ and $\delta$. Take now $y \in X \setminus U$. Then $\|x-y\| > \eps$, hence by \eqref{eq: continuity} $\|\phi_L(x) - \phi_L(y)\| > \delta$. Therefore
	\[ \|x - y\| \leq \diam(X) = \frac{\diam(X)}{\delta} \delta \leq \frac{\diam(X)}{\delta} \|\phi_L(x) - \phi_L(y)\|. \]
	Therefore \eqref{eq: proj biLip} holds for all $y \in X$.
	
	For point \eqref{it: proj detail diffeo}, it again follows from Fubini's theorem, that it suffices to prove that for fixed $x \in M$, for almost every $L \in \LIN(\R^N, \R^k)$, given the open neighbourhood $U$ of $x$ in $M$ satisfying \eqref{eq: proj diffeo}, there exists an open neighbourhood $V$ of $\phi_L(x)$ in $\R^k$ with $V \cap \phi_L(M) \subset \phi_L(U)$. This follows from Fact 1 applied for $X = M$ (we can apply it as $k > \dim M$ implies $\mH^k(M) = 0$). Indeed, for almost every $L \in \LIN(\R^N, \R^k)$ one can choose $\eps > 0$ so that $\overline{B(x, \eps)} \subset U \cap M$ and it follows that $V = B(\phi_L(x), \delta)$ satisfies $V \cap \phi_L(M) \subset \phi_L(U)$, where $\delta = \delta(x, L, \eps) > 0$ is as in Fact 1.
	
	\end{proof}

	Theorem \ref{thm: projections main} follows from Theorem \ref{thm: projections detail} directly by the definition of prevalence and Remark \ref{rem: embed}.
	
	\section{Proof of Theorem \ref{thm: takens main}}
	
	The following is the general version of Theorem \ref{thm: takens main}.
	
	\begin{thm}\label{thm: takens detail}
	Let $M \subset \R^N$ be a compact, embedded $C^r$ submanifold in $\R^N$. Let $T : M \to M$ be a $C^r$ diffeomorphism. Let $\mu$ be a Borel probability measure on $X$. Fix $k \in \N$ and assume $\mu \left( \bigcup \limits_{j=1}^{k-1} \Per_j(T) \right) = 0$. Let $h_1, \ldots, h_m : \R^N \to \R$ be a base of the linear space of $N$-variate polynomials of degree at most $2k-1$. Fix a $C^r$ observable $h : M \to \R$. For $\alpha \in \R^m$, let $\phi_{h_\alpha,k}$ be the $k$-delay coordinate map corresponding to the observable $h_\alpha = h + \sum \limits_{j=1}^m \alpha_j h_j$. The following hold for Lebesgue almost every $\alpha \in \R^m$:
	\begin{enumerate}
		\item\label{it: takens detail biLip} if $\mH^k(X) = 0$ and $k \geq \dim M$, then for $\mu$-a.e. $x$ there exists $C = C(x, \alpha) \in [0, \infty)$ so that
		\begin{equation}\label{eq: takens biLip} \| x- y \| \leq C \|\phi_{h_\alpha,k}(x) - \phi_{h_\alpha,k}(y)\| \text{ holds for every } y \in X,
		\end{equation}
		\item\label{it: takens detail diffeo} if $k > \dim M$, then for $\mu$-a.e. $x$ there exist open neighbourhoods $U$ of $x$ in $M$ and $V$ of $\phi_{h_\alpha,k}(x)$ in $\R^k$ such that $\phi_{h_\alpha,k}$ is a $C^r$ diffeomorphism between $U$ and $V \cap \phi_{h_\alpha,k}(M)$.
	\end{enumerate}
	\end{thm}
	As in the previous section, Theorem \ref{thm: takens main}  follows from Theorem \ref{thm: takens detail} with the use of Remark \ref{rem: embed}. The idea of the proof is similar to the one of Theorem \ref{thm: projections detail} - our goal is to establish analogues of Facts 1. and 2. for delay-coordinate maps with polynomially perturbed observables. This is realized in Propositions \ref{prop: takens injective at point} and \ref{thm: takens diff at point} below. For the rest of this section, we assume that we are in the setting of Theorem \ref{thm: projections detail} and denote for short $\phi_{\alpha} := \phi_{h_\alpha, k}$.
	
	\subsection{Global almost sure continuity} We begin by proving an analogue of Fact 1, formulated as follows.
	\begin{prop}\label{prop: takens injective at point} Assume that $\mH^k(X) = 0$. For almost every $\alpha \in \R^m$, the delay coordinate map $\phi_\alpha$ corresponding to the observable $h_\alpha$ has the property that for $\mu$-a.e. $x \in X \setminus \bigcup \limits_{p=1}^{k-1} \Per_p(T)$, for every $\eps > 0$ there exists $\delta = \delta(x,\alpha, \eps) > 0$ such that
	\begin{equation}\label{eq: continuity}
	\text{for every } y \in X, \text{ if }\|\phi_\alpha(x) - \phi_\alpha(y)\| < \delta, \text{ then } \|x - y\| < \eps.
	\end{equation}	
	\end{prop}
	
	The proof combines methods of the proofs of \cite[Theorem 4.3]{BGS20} and \cite[Theorem 1.21]{BGSPredict}. We shall need the following notation and observations. For $x,y \in M$ we have
	\begin{equation}\label{eq:matrix_form} \phi_\alpha(x) - \phi_\alpha(y) = D_{x,y}\alpha + w_{x,y}
	\end{equation}
	for a $k\times m$ matrix $D_{x,y}$ defined by
	\begin{equation}\label{eq:D_xy}
		D_{x,y} = \begin{bmatrix} h_1(x) - h_1(y) & \ldots & h_m(x) - h_m(y) \\
			h_1(Tx) - h_1(Ty) & \ldots & h_m(Tx) - h_m(Ty) \\
			\vdots & \ddots & \vdots \\
			h_1(T^{k-1}x) - h_1(T^{k-1}y) & \ldots & h_m(T^{k-1}x) - h_m(T^{k-1}y) \\
		\end{bmatrix} 
	\end{equation}
	and a vector $w_{x,y} \in \R^m$
	\[
	w_{x,y} = \begin{bmatrix} h(x) - h(y) \\
		h(Tx) - h(Ty)\\
		\vdots \\
		h(T^{k-1} x) - h(T^{k-1}y) \end{bmatrix}.
	\]
	
	For $x \in M$ we define its \textbf{full orbit} as
	\[ \Orb(x) = \{ T^n x : n \geq 0 \} \cup \{ y \in X : T^n y = x \text{ for some } n \in \N \}.\]
	
	We will use the following lemma.
	
	\begin{lem}[{\cite[Lemma 4.5]{BGS20}}]\label{lem: rank}
		The following hold for $x, y \in M$
		\begin{enumerate}[(i)]
			\item\label{it: rank 1} if $x \neq y$ then $\rank(D_{x,y}) \geq 1$,
			\item\label{it: rank k} if $x \notin \bigcup \limits_{p=1}^{k-1} \Per_p(T)$ and $x \notin \Orb(y)$, then $\rank(D_{x,y})=k$.
		\end{enumerate}
	\end{lem}
	
	\begin{proof}
	The proof of \cite[Lemma 4.5]{BGS20} requires $T$ to be injective and $h_1, \ldots, h_m$ to be a $2k$-interpolating family on $M$. This is so in our case, as we assume $h_1, \ldots, h_m$ to be a linear basis of the linear space of $N$-variate polynomials of degree at most $2k-1$, see \cite[Remark 4.2]{BGS20}.
	\end{proof}
	
	We shall also need the following.
	
\begin{lem}[{\cite[Lemma 14.3]{Rob11}}]\label{lem: key_ineq_inter}
	Let $L : \R^m \to \R^k$ be a linear transformation. Assume that the $p$-th singular value $\sigma_p(L)$ of $L$ is positive for some $p \in \{1, \ldots, k\}$. Then for every $z \in \R^k$ and $r, \eps > 0$,
	\[ \frac{\Leb(\{ \alpha \in B_m(0, r) : \|L \alpha + z \| \leq \eps \})}{\Leb (B_m(0, r))} \leq C_{m,k} \Big(\frac{\eps}{\sigma_p(L) \, r}\Big)^p, \]
	where $C_{m,k} > 0$ is a constant depending only on $m,k$ and $\Leb$ is the Lebesgue measure on $\R^m$.
\end{lem}

	\begin{proof}[{\bf Proof of Proposition \ref{prop: takens injective at point}}]
	By Fubini's theorem, it suffices to fix arbitrary $x \in X \setminus \bigcup \limits_{p=1}^{k-1} \Per_p(T)$ and show that for Lebesgue almost every $\alpha \in \R^m$, for every $\eps>0$ there exists $\delta = \delta(x,\alpha, \eps)>0$ such that \eqref{eq: continuity} holds. For that, it is enough to prove that for almost every $\alpha \in \R^m$
	\begin{equation}\label{eq: injectivity} \text{for every } y \in X, \text{ if } \phi_\alpha(x) = \phi_\alpha(y), \text{ then } x = y.
	\end{equation}
	Indeed, if $\alpha \in \R^m$ is such that \eqref{eq: injectivity} holds and there exists $\eps > 0$ such that for every $n \in \N$ there exists $y_n \in X$ with $\|\phi_\alpha(x) - \phi_\alpha(y_n)\| < 1/n$ and $\|x - y\| \geq \eps$, then by passing to the limit of a convergent subsequence of $(y_n)_{n=1}^\infty$ (recall that we assume $X$ to be compact) we see that there exists $y \in X$ with $\phi_\alpha(x) = \phi_\alpha(y)$ and $\| x -y  \| \geq \eps$ (by continuity of $\phi_\alpha$), contradicting $\eqref{eq: injectivity}$.
	
	The proof of \eqref{eq: injectivity} is essentially the same as the proof of \cite[Theorem 4.3]{BGS20}, with minor adjustments.\footnote{one has to note that in our case we can take $\tilde{X} = X$ and then the proof of $\eta_m(A_x) = 0$ follows in a similar manner as in \cite{BGS20} with the use of Lemmas \ref{lem: rank} and \ref{lem: key_ineq_inter}.} We provide details for the convenience of the reader.
	
	Fix arbitrary $r > 0$. We shall prove that
	\begin{equation}\label{eq: rho goal}
		\Leb \left( \left\{ \alpha \in B_m(0, r) : \underset{ y \in X \setminus \{x\}}{\exists}\ \phi_\alpha(x) = \phi_\alpha (y) \right\} \right) = 0.
	\end{equation}
	
	By injectivty of $T$ we see that the set $E := \Orb(x) \setminus \{ x\}$ is at most countable. Lemma \ref{lem: rank}.\ref{it: rank 1} gives that $\rank(D_{x,y}) \geq 1$ for $y \in E$, hence $D_{x,y} \neq 0$. Therefore by \eqref{eq:matrix_form} we have
	\begin{equation}\label{eq: bound on E}
	\begin{split}	\Leb & \left( \left\{ \alpha \in B_m(0, r) : \underset{ y \in E}{\exists}\ \phi_\alpha(x) = \phi_\alpha (y) \right\} \right) \\
	& = \sum \limits_{y \in E} \Leb \left( \left\{ \alpha \in B_m(0, r) : D_{x,y}\alpha = -w_{x,y} \right\} \right) \\
	& = 0,
	\end{split}
	\end{equation}
	as each set $\{ \alpha \in \R^m : D_{x,y}\alpha = -w_{x,y}\}$ is a proper affine subspace of $\R^m$. Note further that $x \in \Orb(y)$ if and only if $y \in \Orb(x)$ and hence Lemma \ref{lem: rank}.\ref{it: rank k} gives that $\rank(D_{x,y}) = k$ for $y \in F:= X \setminus \Orb(x)$, hence also $\sigma_k(D_{x,y}) > 0$ for $y \in F$. Therefore
	\begin{equation}\label{eq: F decomp} F = \bigcup \limits_{n=1}^\infty F_n, \text{ where } F_n:= \{ y \in F : \sigma_k(D_{x,y}) \geq 1/n\}.
	\end{equation}
	Fix arbitrarily small $\eps >0$ and consider a countable cover $F_n \subset \bigcup \limits_{i=1}^\infty B(y_i, \eps_i)$, where $y_i \in F_n$ and $\sum \limits_{i=1}^\infty \eps_i^k < \eps$. Such a cover exists as $\mH^k(F_n) \leq \mH^k(X) = 0$ by assumption. As $T$ and $h, h_1, \ldots, h_m$ are $C^r$ maps on a compact manifold, they are also Lipschitz and hence there exists $Q = Q(r, T, h, h_1, \ldots, h_m )$ such that
	\[ \text{ if } \phi_\alpha(x) = \phi_\alpha (y) \text{ for some } y \in F_n, \text{ then } \|\phi_\alpha(x) - \phi_\alpha (y_i)\| \leq Q\eps_i \text{ for some } i \in \N. \]
	Combining this with Lemma \ref{lem: key_ineq_inter} and \eqref{eq:matrix_form} gives
	\[
	\begin{split}
	\Leb & \left( \left\{ \alpha \in B_m(0, r) : \underset{ y \in F_n}{\exists}\ \phi_\alpha(x) = \phi_\alpha (y) \right\} \right) \\
	& \leq \sum \limits_{i=1}^\infty \Leb \left( \left\{ \alpha \in B_m(0, r) : \| \phi_\alpha(x) - \phi_\alpha (y_i)\| \leq Q \eps_i \right\} \right) \\
	& \leq \sum \limits_{i=1}^\infty \Leb \left( \left\{ \alpha \in B_m(0, r) : \| D_{x,y_i} \alpha + w_{x,y_i} \| \leq Q\eps_i \right\} \right) \\
	& \leq  C_{m,k} \Leb (B_m(0, r)) r^{-k}n^k Q^k \sum \limits_{i=1}^\infty \eps_i^k \\
	& \leq C_{m,k} \Leb (B_m(0, r)) r^{-k}n^k Q^k \eps.
	\end{split}
	\]
	As $\eps>0$ can be chosen arbitrarily small, we see that
	\[ \Leb \left( \left\{ \alpha \in B_m(0, r) : \underset{ y \in F_n}{\exists}\ \phi_\alpha(x) = \phi_\alpha (y) \right\} \right) = 0,\]
	hence by \eqref{eq: F decomp}
	\[\Leb \left( \left\{ \alpha \in B_m(0, r) : \underset{ y \in F}{\exists}\ \phi_\alpha(x) = \phi_\alpha (y) \right\} \right) = 0.\]
	Since $X \setminus \{x\} \subset E \cup F$, combining this with \eqref{eq: bound on E} yields \eqref{eq: rho goal}. As $r > 0$ can be taken arbitrarily large, we conclude that \eqref{eq: injectivity} holds for Lebesgue almost every $\alpha \in \R^m$. As already discussed, this finishes the proof.
	\end{proof}
	
	\subsection{Almost sure local immersion} Now we turn to the time-delayed version of Fact 2.
	
	\begin{prop}\label{thm: takens diff at point}
	Assume that  $k \geq \dim M$. Then for almost every $\alpha \in \R^m$, the following holds for $\mu$-a.e. $x \in M \setminus \bigcup \limits_{p=1}^{k-1} \Per_p (T)$: there exists an open neighbourhood $U$ of $x$ in $M$ such that $\phi_{\alpha}$ restricted to $U$ is a $C^r$ diffeomorphism onto its image.
	\end{prop}
	
	For the proof we will need some tools from the interpolation theory.
	\begin{lem}[{\cite[Lemma 4.1]{SYC91}}]\label{lem: RN interpolation}
		For every collection $z_1, z_2, \ldots z_k$ of $k$ distinct points in $\R^N$ and vectors $u_1, \ldots, u_k \in \R^N$ there exists an $N$-variate polynomial $p$ of degree at most $k$ such that
		\[ \nabla p(z_i) = u_i \text{ for every } i \in \{1, \ldots, k\}.\]
	\end{lem}
	
	\begin{cor}\label{cor: manifold interpolation}
	Fix distinct points $z_1, \ldots, z_k \in M$ and linear maps $L_i \in \Lin(T_{z_i}M, \R),\ i =1, \ldots, k$. Then, there exists an $N$-variate polynomial $p$ of degree at most $k$ such that
	\[ D_{z_i}(p|_M) = L_i \text{ for every } i \in \{1, \ldots, k \}.\]
	\end{cor}
	
	\begin{proof} As $M$ is an embedded $C^r$ submanifold in $\R^N$, we can identify each $T_{z_i} M$ with a $\dim M$-dimensional linear subspace of $\R^N$. Then for all $C^1$ functions $f : \R^N \to \R$ one has
		\[ D_{z_i}(f|_M) = (D_{z_i}f)|_{T_{z_i} M} \]
		(here $D_{z_i}(f|_M) \in \LIN(T_{z_i} M, \R)$ is the differential of $f|_M$, while $(D_{z_i}f)|_{T_{z_i} M} \in \LIN(T_{z_i} M, \R)$ is the restriction to $T_{z_i} M$ of the differential $D_{z_i} f \in \LIN(\R^N, \R)$). Therefore, it suffices to show that there exists a polynomial $p : \R^N \to \R$ of degree at most $2k-1$ such that
		\begin{equation}\label{eq: finding diff}
			(D_{z_i}p)|_{T_{z_i} M} = L_i.
		\end{equation}
		Take $u_i \in \R^N$ so that $\langle u_i, v \rangle = L_i v$ for all $v \in T_{z_i} M$ (e.g. one can extend $L_i \in \LIN(T_{z_i} M, \R)$ in an arbitrary fashion to $A_i \in \LIN(\R^N, \R)$ and take $u_i$ to be the preimage of $A_i$ under the canonical identification $\R^N \simeq \LIN(\R^N,\R)$ via $u \mapsto \langle u, \cdot \rangle$). Let $p$ be a polynomial as in Lemma \ref{lem: RN interpolation} satisfying $\nabla p(z_i) = u_i$. As $D_{z_i} p (v) = \langle \nabla p (z_i), v \rangle$ for all $v \in T_{z_i}\R^N = \R^N$, we see that $p$ satisfies \eqref{eq: finding diff}.
		
	\end{proof}
	
	\begin{proof}[{\bf Proof of Proposition \ref{thm: takens diff at point}}]
	By Fubini's theorem it suffices to prove that for arbitrary fixed $x \in M \setminus \bigcup \limits_{p=1}^{k-1} \Per_p (T)$, for Lebesgue almost every $\alpha \in \R^m$ there exists an open neighbourhood $U$ of $x$ in $M$ such that $\phi_{h_{\alpha}}$ restricted to $U$ is a $C^r$ diffeomorphism onto its image. By Lemma \ref{lem: immersion} it suffices to prove that for almost every $\alpha \in \R^m$ the differential $D_x \phi_\alpha \in \LIN(T_x M, \R^k)$ is of full rank, i.e. $\rank(D_x \phi_\alpha) = d$, where $d = \dim M$. Consider a linear map $\Psi : \R^m \to \LIN(T_x M, \R^k)$ given by $\Psi \alpha  = \sum \limits_{j=1}^m \alpha_j D_x \phi_{h_j}$. Note that $D_x \phi_\alpha = D_x \phi_h + \Psi \alpha$. Therefore our goal is to show that for almost every $\alpha \in \R^m$ it holds that $\Psi \alpha \in G := \{ L \in \LIN(T_x M, \R^k) : \rank(D_x \phi_h + L) = d \}$. By Lemma \ref{lem: full rank}, $G$ is a full Lebesgue measure subset of the $dk$-dimensional linear space $\LIN(T_x M, \R^k)$. It therefore suffices to show that $\Psi$ is surjective, as then $\Psi\Leb_m$ is absolutely continuous with respect to the Lebesgue measure on $\LIN(T_x M, \R^k)$ and hence $\Leb_m(\Psi^{-1}(G^c)) = 0$. Fix arbitrary $L \in \LIN(T_x M, \R^k)$. Our goal is to prove that there exists an $N$-variate polynomial $p$ of degree at most $2k-1$ such that $D_x \phi_p = L$ (as then there exists $\alpha \in \R^m$ such that $p = \sum \limits_{j=1}^m \alpha_j h_j$ and hence $\Psi \alpha  = D_x \phi_p = L$). Note that by the chain rule
	\begin{equation}\label{eq: phi diff chain}
	\begin{split} D_x \phi_p & = (D_x ( p|_M ), D_x (p|_M \circ T), \ldots, D_x (p|_M \circ T^{k-1})) \\
	& = (D_x (p|_M), D_{Tx} (p|_M) \circ D_x T, \ldots, D_{T^{k-1}x} (p|_M) \circ D_{T^{k-2} x} T \circ \cdots \circ D_x T).
	\end{split}
	\end{equation}
	Write $L = (L_1, \ldots, L_k)$ with $L_1, \ldots, L_k \in \LIN(T_x M, \R)$. As $x \notin \bigcup \limits_{p=1}^{k-1} \Per_p (T)$, we can invoke Corollary \ref{cor: manifold interpolation} to conclude that here exists a polynomial $p$ of degree at most $2k-1$ such that
	\[ D_{T^i x} (p|_M) = L_i \circ (D_{T^{i-1}x}T \circ \cdots \circ D_x T)^{-1} \text{ for } i =0, \ldots, k-1.\]
	Combining this with \eqref{eq: phi diff chain} proves that $D_x \phi_p = L$. As $L$ was arbitrary we see that $\Psi$ is surjective, concluding the proof.
	\end{proof}
	
	\begin{proof}[{\bf Proof of Theorem \ref{thm: takens detail}}]
	The proof follows in exactly the same way as the proof of Theorem \ref{thm: projections detail}, with maps $\phi_L$ replaced by $\phi_\alpha$, Fact 1 replaced by Proposition \ref{prop: takens injective at point} and Fact 2 replaced by Proposition \ref{thm: takens diff at point}. Note that we use here the assumption $\mu\left( \bigcup \limits_{p=1}^{k-1} \Per_p(T) \right) = 0$.
	\end{proof}
	
	\section{Applications}\label{sec: applications}
	
	\subsection{Bounds on the prediction errors}\label{subsec: pred errors}
	
	The idea that one can reduce the embedding dimension twice with respect to classical Takens theorem and still obtain statistically reliable model of the system goes back to the work of Schroer, Sauer, Ott and Yorke \cite{SSOY98}. Instead of studying directly the almost-sure properties of delay-coordinate maps, they have defined certain prediction errors and formulated conjectures on their decay rate. More precisely, given a measure $\mu$ on $X$ and an observable $h : X \to \R$ with the corresponding delay-coordinate map $\phi_{h,k} : X \to \R^k$, one defines for $y \in \supp (\phi_{h,k}\mu)$ and $\eps>0$
	\begin{align*}
		\chi_{h,\eps}(y) &=  \frac{1}{\mu(\phi_{h,k}^{-1}(B(y, \eps)))} \int \limits_{\phi_{h,k}^{-1}(B(y, \eps))} \phi_{h,k}\circ T \: d\mu,\\
		\sigma_{h,\eps}(y) &= 
		\bigg(\frac{1}{\mu(\phi_{h,k}^{-1}(B(y, \eps)))} \int \limits_{\phi_{h,k}^{-1}(B(y, \eps))} \|\phi_{h,k} \circ T -\chi_{h,\eps}(y)\|^2d\mu\bigg)^{\frac{1}{2}}.
	\end{align*}
	The quantity $\sigma_{h,\eps}(y)$ measures the mean-square error of the prediction of the one-step future $\phi_{h,k}(Tx)$ if the current state $y = \phi_{h,k}(x)$ is given with a random noise (governed by $\mu$) up to precision $\eps$. The \textbf{prediction error} is defined as the limit $\sigma_h(y) = \lim \limits_{\eps \to 0} \sigma_{h,\eps}(y)$, provided that the limit exists. It was proven in \cite[Theorem 1.17]{BGSPredict} that $\sigma_h(y) = 0$ for $\phi_{h,k}\mu$-a.e. $y$ if and only if $h$ is almost surely deterministically $k$-predictable, i.e. there exists a full $\mu$-measure set $X_h$ such that $\phi_{h,k}(x) = \phi_{h,k}(y)$ implies $\phi_{h,k}(Tx) = \phi_{h,k}(Ty)$. Note that it follows from Theorem \ref{thm: takens main} that this is the case for a prevalent $C^r$ observable provided that $k > \hdim X$ and $k \geq \dim M$, as then $\phi_{h,k}$ is injective on a full-measure set. In fact, weaker assumptions suffice: \cite[Theorem 1.21]{BGSPredict} proves that that the same holds for if $k > \hdim \mu$. The \textit{Schorer-Sauer-Ott-Yorke predicition error conjecture} \cite{SSOY98} postulates more precise bounds on the decay rate for the convergence $\lim \limits_{\eps \to 0} \sigma_{h,\eps}(y) = 0$ in probability. After suitable modifications, these conjectures were essentially established in a series of works \cite{BGS22, BGSPredict, BGSLimits}, see therein for details and see \cite[Theorem 1.16]{BGSLimits} for the summary of the results on the above conjecture. In particular, in the case $k > \udim X$ the bound
	\[ \mu\left( \left\{ x \in X : \sigma_{h,\eps}(\phi_{h,k}(x)) > \delta \right\} \right) \leq C_{\delta, \theta} \eps^{k - \udim(X) - \theta} \]
	was obtained for prevalent Lipschitz observables $h$ and every $\delta, \theta > 0$. With the aid of Theorem \ref{thm: takens main} we can improve this by obtaining almost sure pointwise (rather than in probability) bounds on $\sigma_{h,\eps}(\phi_{h,k}(x))$.
	
	\begin{thm}\label{thm: prediction error}
		Let $M$ be a compact, Riemannian $C^r$ manifold. Let $T : M \to M$ be a $C^r$ diffeomorphism. Let $X \subset M$ be a compact set and $\mu$ be a Borel probability measure on $X$. Fix $k \in \N$ such that $\mH^k(X) = 0$ and $k \geq \dim M$ and assume that $\mu \left( \bigcup \limits_{p=1}^{k-1} \Per_p(T) \right) = 0$. The following hold for a prevalent $C^r$ observable $h : M \to \R$: for $\mu$-a.e. $x \in X$ there exists $C=C(x,h)$ such that for all $\eps>0$
		\begin{equation}\label{eq: sigma pointwise bound}
			\sigma_{h,\eps}(\phi_{h,k}(x)) \leq C\eps.
		\end{equation}
	\end{thm}
	
	\begin{proof}
	By Remark \ref{rem: embed} can assume that $M$ is embedded into $\R^N$. Applying Theorem \ref{thm: takens detail} we obtain that a prevalent $C^r$ observable $h : M \to \R$ is such that for $\mu$-a.e. $x \in X$ there exists $C=C(x,h)$ with
		\begin{equation}\label{eq: takens biLip2} \| x- y \| \leq C \|\phi_{h,k}(x) - \phi_{h,k}(y)\| \text{ for every } y \in X,
	\end{equation}
	Fix $h$ and $x$ such that the above holds. We shall prove that \eqref{eq: sigma pointwise bound} holds at $x$. Note first that \eqref{eq: takens biLip2} gives that
	\[X \cap \phi_{h,k}^{-1}(B(\phi_{h,k}(x),\eps)) \subset B(x, C\eps).\]
	As $\phi_{h,k}$ and $T$ are Lipschitz, we have therefore
	\[ \|\chi_{h,\eps}(\phi_{h,k}(x)) - \phi_{h,k}(Tx)\| \leq C\Lip(\phi_{h,k} \circ T)\eps \]
	and hence similarly
	\[
	\begin{split}
		\sigma_{h,\eps}(\phi_{h,k}(x)) & \leq \|\chi_{h,\eps}(\phi_{h,k}(x)) - \phi_{h,k}(Tx)\| \\
		& \quad + \bigg(\frac{1}{\mu(\phi_{h,k}^{-1}(B(\phi_{h,k}(x), \eps)))} \int \limits_{\phi_{h,k}^{-1}(B(\phi_{h,k}(x), \eps))} \|\phi_{h,k} \circ T -\phi_{h,k}(Tx)\|^2d\mu\bigg)^{\frac{1}{2}} \\
		& \leq 2 C \Lip(\phi_{h,k} \circ T)\eps.
	\end{split}
	\]
	\end{proof}
	
	\subsection{Preservation of the Lyapunov exponents}\label{sec: Lyap exp}
	
	Let $T:M \to M$ be a $C^r$ diffeomorphism of a compact, Riemannian $C^r$ manifold $M$. Let $\mu$ be a $T$-invariant and ergodic Borel probability measure on $M$. Then there exist $\ell \in \N$, numbers $d_1, \ldots, d_\ell \in \N$ such that $\sum \limits_{i=1}^\ell d_i = \dim M$ and real numbers $\chi_1 < \ldots < \chi_\ell$ such that for $\mu$-a.e. $x \in M$ there exists a decomposition
	\begin{equation}\label{eq: lyap splitting} T_x M = \bigoplus \limits_{i=1}^\ell H_i(x)
	\end{equation}
	into linear subspaces $H_i(x)$ satisfying $\dim H_i(x) = d_i$ and such that
	\begin{equation}\label{eq: lyap conv} \lim \limits_{n \to \infty} \frac{1}{n} \log \frac{\|D_x(T^n) v\|}{\|v\|}  = \chi_i \text{ uniformly in } v \in H_i(x) \setminus \{0\}, \text{ for } i = 1, \ldots, \ell.
	\end{equation}
	This is a consequence of the Oseledets' multiplicative ergodic theorem, see e.g. \cite[Supplement Section 2]{KatokHasselblatt} for more details. The numbers $\chi_1, \ldots, \chi_\ell$ are called the \textbf{Lyapunov exponents} of $\mu$ and $d_1, \ldots, d_\ell$ are their multiplicities \cite{BarreiraPesinLyapSmooth}. Lyapunov exponents describe the expansion in the system, quantifying its chaotic behaviour. Estimating them from a time series is an important problem in applications \cite{ER85, EKORCLyapTimeSeries} with its algorithmic aspects being studied mathematically \cite{MeraMoranERAlg}.
	
	As an application of Theorem \ref{thm: takens main}.\eqref{it: takens main dim} we can prove that if $k > \dim M$, then a typical orbit of the observed dynamics on $\phi_{h,k}(M)$ will approximate the Lyapunov exponents of the original system $(M,T)$ with full frequency. To explain this more precisely, assume that $h : M \to R$ is such that point \eqref{it: takens main dim} of Theorem \ref{thm: takens main} holds. Then, while $\phi_{h,k}(M)$ in general will not be a manifold, it will be locally so at $\phi_{h,k}\mu$-a.e. $z \in \phi_{h,k}(M)$ in the sense that there exists a neighbourhood $V_z$ of $z$ in $\R^k$ such that $V_z \cap \phi_{h,k}(M) \subset \R^k$ is an embedded $\dim M$-dimensional $C^r$ submanifold in $\R^k$ (as by Theorem \ref{thm: takens main}, $\phi_{h,k}$ gives a local embedding of $M$ into $\R^k$ at $\mu$-a.e. $x \in M$, see \cite[Chapter 5]{LeeIntroSmoothManifolds}). Therefore, we can treat the tangent spaces $T_y (V_z \cap \phi_{h,k}(M))$ for $y \in V_z \cap \phi_{h,k}(M)$ as linear subspaces of $\R^k$. For $y=z$ the tangent space $T_z (V_z \cap \phi_{h,k}(M))$ does not depend on the choice of the neighbourhood $V_z$ (provided that $V_z$ is small enough), so we will denote it shortly by $T_z \phi_{h,k}(M)$. If we assume that $\mu$ is $T$-invariant, then there exists a $T$-invariant full $\mu$-measure Borel set $X_h \subset M$ such that the above can be carried out for every $z \in \phi_{h,k}(X_h)$ and $\phi_{h,k}$ is injective on $X_h$ (as every full-measure set contains a full-measure invariant subset). The inverse map $\phi_{h,k}^{-1} : \phi_{h,k}(X_h) \to X_h$ has the property that for every $x \in X_h$ there exists a open neighbourhood $V_z$ of $z = \phi_{h,k}(x)$ in $\R^k$ such that $\phi_{h,k}^{-1}$ extends to a map $\phi_{h,k}^{-1} : V_z \cap \phi_{h,k}(M) \to M$ which is a $C^r$ diffeomorphism onto its image. In particular the differential $D_{z} \phi_{h,k}^{-1} : T_z \phi_{h,k}(M) \to T_x M$ is a linear isomorphism. Its inverse equals $D_{x} \phi_{h,k} : T_x M \to T_z \phi_{h,k}(M)$ (note that indeed $D_{x} \phi_{h,k} : T_x M \to T_z \R^k$ is injective and its image equals $T_z \phi_{h,k}(M)$). We can therefore define the prediction map $S_{h,k} : \phi_{h,k}(X_h) \to \phi_{h,k}(X_h)$ as in \eqref{eq: prediction map} and it follows that $S_{h,k}$ is $C^r$ differentiable at every $z \in \phi_{h,k}(X_h)$ in the sense that there exist open neighbourhoods $V_z$ of $z$ and $V_{S_{h,k}(z)}$ of $S_{h,k}(z)$ in $\R^k$ such that $S_{h,k}$ extends to a $C^r$ diffeomorphism of $V_z \cap \phi_{h,k}(M)$ onto $V_{S_{h,k}(z)} \cap \phi_{h,k}(M)$. Again, the differential $D_z S_{h,k} : T_z  \phi_{h,k}(M) \to T_{S_{h,k}(z)} \phi_{h,k}(M)$ is a linear isomorphism. As $\phi_{h,k}(X_h)$ is $S_{h,k}$-invariant, we can take arbitrary iterates of $S_{h,k}$ and consider $ \frac{1}{n} \log \frac{\|D_{z}(S_{h,k}^n)v\|}{\|v\|}$ as the Lyapunov exponent observed in finite time. We can prove that these observed Lyapunov exponents approximate the Lyapunov exponents of the original system in the following sense.
	
	\begin{thm}\label{thm: lyap exp}
		Let $M$ be a compact, Riemannian $C^r$ manifold. Let $T : M \to M$ be a $C^r$ diffeomorphism and let $\mu$ be a $T$-invariant, ergodic Borel probability measure on $M$, having Lyapunov exponents $\chi_1, \ldots, \chi_\ell$ with multiplicities $d_1, \ldots, d_\ell$. Fix $k > \dim M$. Then for a prevalent $C^r$ observable $h : M \to \R$, for $\mu$-a.e. $x \in M$ there exists a decomposition $T_{\phi_{h,k}(x)} \phi_{h,k}(M) = \bigoplus \limits_{i=1}^\ell \tilde{H_i}(\phi_{h,k}(x))$ into linear subspaces $\tilde{H_i}(\phi_{h,k}(x))$ with $\dim \tilde{H_i}(\phi_{h,k}(x)) = d_i$ and such that the following property holds: for every $\eps > 0$
		
		\[ \lim \limits_{N \to \infty} \frac{1}{N} \# \left\{ 0 \leq n < N : \sup \limits_{v \in \tilde{H_i}(x) \setminus \{ 0 \}} \left| \frac{1}{n} \log \frac{\|D_{\phi_{h,k}(x)}(S_{h,k}^n)v\|}{\|v\|} - \chi_i \right| < \eps \text{ for all }\ 1 \leq i \leq \ell \right\} = 1,\]
		where $S_{h,k}$ is the prediction map defined as in \eqref{eq: prediction map}.
		
	\end{thm}
	
	\begin{proof}
	Take a prevalent $C^r$ observable $h : M \to \R$ such that point \eqref{it: takens main dim} of Theorem \ref{thm: takens main} holds. Let $X_h \subset M$ be a full $\mu$-measure and $T$-invariant set as described in the previous paragraph and such that \eqref{eq: lyap splitting} and \eqref{eq: lyap conv} hold for every $x \in X_h$. We have then that
	\begin{equation}\label{eq: S equivariance}
	S^n_{h,k} = \phi_{h,k} \circ T^n \circ \phi_{h,k}^{-1} \text{ on } \phi_{h,k}(X_h).
	\end{equation}
	Set $\tilde{H_i}(\phi_{h,k}(x)) := D_x{\phi_{h,k}}(H_i(x))$. As $D_x \phi_{h,k} : T_{x}M \to T_{\phi_{h,k}}\phi_{h,k}(M)$ is a linear isomorphism, we obtain from \eqref{eq: lyap splitting} that $T_{\phi_{h,k}(x)} \phi_{h,k}(M) = \bigoplus \limits_{i=1}^\ell \tilde{H_i}(\phi_{h,k}(x))$ and $\dim \tilde{H_i} = d_i$. For $M > 0$ set
	\[ E_M = \left\{ x \in X_h : \|D_x \phi_{h,k}\| \leq M,\ \|D_{\phi_{h,k}(x)} \phi_{h,k}^{-1}\| \leq M \right\}, \]
	where the operator norm is induced by the scalar products on $T_x M$ and $T_{\phi_{h,k}(x)}\phi_{h,k}(M)$ coming form the respective Riemannian metrics. Note that
	\begin{equation}\label{eq: E measure}
		\lim \limits_{M \to \infty} \mu(E_M) = \mu(X_h) = 1.
	\end{equation}
	By the Birkhoff ergodic theorem, for $\mu$-a.e. $x \in X_h$
	\begin{equation}\label{eq: E birkhoff} \lim \limits_{N \to \infty } \frac{1}{N} \# \left\{ 0 \leq n < N : T^n x \in E_M\right\} = \mu(E_M).
	\end{equation}
	Take $x \in X_h$ and $n \geq 1$ such that $T^n x \in E_M$. Then by \eqref{eq: S equivariance} and the chain rule, for $v \in T_{\phi_{h,k}(x)} \phi_{h,k}(M)$
	\[
		\|D_{\phi_{h,k}(x)}(S_{h,k}^n)v\| = \|D_{\phi_{h,k}(x)}(\phi_{h,k} \circ T^n \circ \phi_{h,k}^{-1})v\| = \|\left( D_{T^n x}\phi_{h,k} \circ D_x T^n \circ D_{\phi_{h,k}(x)}\phi_{h,k}^{-1} \right) v \| \]
	and hence
	\[ \frac{1}{M} \| D_x T^n  \left( D_{\phi_{h,k}(x)}\phi_{h,k}^{-1} (v) \right) \| \leq \|D_{\phi_{h,k}(x)}(S_{h,k}^n)v\| \leq M \| D_x T^n \left( D_{\phi_{h,k}(x)}\phi_{h,k}^{-1} (v) \right) \|. \]
	Therefore, for $v \in T_{\phi_{h,k}(x)} \phi_{h,k}(M) \setminus \{ 0 \}$
	\[ \left| \frac{1}{n} \log \frac{\|D_{\phi_{h,k}(x)}(S_{h,k}^n)v\|}{\|v\|} - \frac{1}{n} \log \frac{\|D_x T^n \left( D_{\phi_{h,k}(x)}\phi_{h,k}^{-1} (v)\right)\|}{\|v\|} \right| \leq \frac{\log M}{n}.\]
	As for $v \in \tilde{H_i}(\phi_{h,k}(x))$ we have $D_{\phi_{h,k}(x)}\phi_{h,k}^{-1} (v) \in H_i(x)$ (since $D_{\phi_{h,k}(x)}\phi_{h,k}^{-1}$ is the inverse of $D_x \phi_{h,k} : T_x M \to T_{\phi_{h,k}(x)}\phi_{h,k}(M)$), we see that if $T^n x \in E_M$, then
	\[ \sup \limits_{v \in \tilde{H_i}(x) \setminus \{ 0 \}} \left| \frac{1}{n} \log \frac{\|D_{\phi_{h,k}(x)}(S_{h,k}^n)v\|}{\|v\|} - \chi_i \right| \leq
	\sup \limits_{v \in H_i(x) \setminus \{ 0 \}} \left| \frac{1}{n} \log \frac{\|D_{x}(T^n) v\|}{\|v\|} - \chi_i \right| + \frac{\log M}{n}. \]
	Since we assume that \eqref{eq: lyap conv} holds for every $x \in X_h$, we see that for fixed $\eps > 0$ and any $M > 0$, there exists $n_0 = n_0(x,M,\eps)$ such that
	\[ \sup \limits_{v \in \tilde{H_i}(x) \setminus \{ 0 \}} \left| \frac{1}{n} \log \frac{\|D_{\phi_{h,k}(x)}(S_{h,k}^n)v\|}{\|v\|} - \chi_i \right| < \eps \text{ if } n \geq n_0 \text{ and } T^n x \in E_M. \]
	Therefore, for such $x \in X_h$ and every $M > 0$
	\[\begin{split}
		\lim \limits_{N \to \infty} & \ \frac{1}{N} \# \left\{ 0 \leq n < N : \sup \limits_{v \in \tilde{H_i}(x) \setminus \{ 0 \}} \left| \frac{1}{n} \log \frac{\|D_{\phi_{h,k}(x)}(S_{h,k}^n)v\|}{\|v\|} - \chi_i \right| < \eps \text{ for all }\ 1 \leq i \leq \ell \right\} \\
		&  \geq  \lim \limits_{N \to \infty } \frac{1}{N} \left\{ 0 \leq n < N : T^n x \in E_M\right\}.
	\end{split}\]
	Combining this with \eqref{eq: E birkhoff} and \eqref{eq: E measure} finishes the proof.
	\end{proof}
		
		\bibliographystyle{alpha}
		\bibliography{universal_bib}
		
	\end{document}